\documentclass[noamsfonts]{amsart}
\usepackage{setspace}
\usepackage[charter,expert]{mathdesign}
\usepackage[osf]{XCharter}
\selectfont
\usepackage[activate={true,nocompatibility},final]{microtype}

\usepackage[all, arc]{xy}
\SelectTips{eu}{}
\entrymodifiers={+!!<0pt,\fontdimen22\textfont2>}

\usepackage[pdftex, colorlinks=true, linkcolor=blue, citecolor=magenta, linktocpage]{hyperref}

\theoremstyle{plain}

\newtheorem{prop}[subsubsection]{Proposition}
\newtheorem{cor}[subsubsection]{Corollary}

\theoremstyle{definition}
\newtheorem{example}[subsubsection]{Example}

\theoremstyle{definition}

\numberwithin{equation}{subsubsection}

\def\cC{\mathcal{C}}

\def\11{\mathbf{1}}

\def\CC{\mathbf{C}}

\def\PP{\mathbf{P}} 
 
\def\RR{\mathbf{R}}

\def\ZZ{\mathbf{Z}}

\def\et{\text{\'et}}

\def\Hom{\mathrm{Hom}}

\def\id{\mathrm{id}}

\def\Spec{\mathrm{Spec}}

\def\Map{\mathrm{Map}}
\def\PreGrpd{\mathrm{PreGrpd}}

\def\Ob{\mathrm{Ob}}
\def\Mor{\mathrm{Mor}}
\def\Grpd{\mathrm{Grpd}}

\def\Pre{\mathrm{Pre}}

\def\E{\mathbb{E}}
\def\B{\mathbb{B}}

\def\Aut{\mathrm{Aut}}
\def\i{\iota}
\def\Op{\mathrm{Op}}

\newcommand{\mapright}[1]{\xrightarrow{#1}}

\title{Fixed points and Langlands parameter spaces for real groups}
\author{R. Virk}
\address{The Appalachians}
\begin{document}
\maketitle
\renewcommand{\thesubsection}{\textbf{\arabic{section}.\arabic{subsection}}}
\section{Introduction}
This note records that the Langlands parameter spaces, associated by Adams-Barbasch-Vogan to a real group \cite{ABV}, may be described as homotopy fixed points (fixed point stacks) of the spaces associated to the corresponding complex group. The reader should compare with \cite[Section 4]{BN} and also consult \cite[Section 2-3]{T}.

\subsection{Conventions}
The symbol $\Gamma$ is reserved for the group $\ZZ/2\ZZ$. The action of the non-trivial element on a $\Gamma$-set is denoted $x\mapsto \bar x$.
If a category has a final object, set
\[ \ast = \text{final object}. \]
A morphism or map between categories will mean a functor of categories.

\section{Groupoids}
A groupoid is a (small) category $X$ in which all morphisms are invertible.
The set of objects of $X$ is denoted $\Ob(X)$, the set of morphisms $\Mor(X)$.
A map of groupoids $f\colon X\to Y$ is called a \emph{weak equivalence} if it is an equivalence of categories. 
It is called a \emph{fibration} if, for each $x\in\Ob(X)$ and each isomorphism $\alpha\colon f(x) \to y$ in $Y$, there exists an isomorphism $\beta\colon x\to x_1$ in $X$, such that $f(\beta) = \alpha$. Fibrations that are also weak equivalences are called \emph{acyclic} fibrations.

\begin{example}\label{eg:BG}
    If $G$ is a group acting on a set $X$, write $\E_GX$ for the groupoid with $\Ob(\E_GX) = X$ and $\Mor(\E_G X)=G\times X$. Structure maps are given by the action $G\times X \to X$ (target of a morphism), the projection $G\times X \to X$ (source of a morphism), the identity section $X \to G \times X$ (identity map), and composition of morphisms given by group multiplication. 
Set
\[ \E G = \E_G G \quad \text{and} \quad \B G = \E_G(\ast). \]
I.e., the groupoids associated to $G$ acting on itself (left multiplication), and $G$ acting trivially on a single element set, respectively.
The map $\E G \to \ast$ is an acyclic fibration. More generally, if $N\subset G$ is a normal subgroup whose action on $X$ is free, then the evident canonical map $\E_{G} X \to \E_{G/N}X/N$ is an acyclic fibration.
%
\end{example}

\subsection{Homotopy fixed points}
An action of a group on a groupoid is the data of actions on sets of objects and morphisms, such that the structure maps of the groupoid are equivariant with respect to these actions. An equivariant map between groupoids equipped with such actions is a map of groupoids that induces equivariant maps on sets of objects and morphisms.

Let $X$ be a groupoid equipped with a $\Gamma=\ZZ/2\ZZ$ action. The groupoid $X^{h\Gamma}$ of \emph{homotopy fixed points} is defined by
\[ \Ob(X^{h\Gamma}) = \{ (x,\phi)\in \Ob(X)\times \Mor(X) \mid \phi\in \Hom(x,\bar x) \text{ and } \bar\phi = \phi^{-1}\}, \]
with an arrow $(x,\phi) \to (x_1, \phi_1)$ for each $\alpha\in\Hom(x,x_1)$ such that $\phi_1\alpha = \bar\alpha\phi$ (reminder: $\alpha\mapsto \bar\alpha$ denotes the action of the non-trivial element of $\Gamma$).

\begin{example}Let $X$ be a set with $\Gamma$-action. View $X$ as a groupoid (the only morphisms are the identity maps). Then $X^{h\Gamma}$ is the usual set of fixed points $X^{\Gamma}$.
\end{example}

\begin{example}
Let $X$ be a groupoid. Consider the trivial action of $\Gamma$ on $X$. Then
\[ \Ob(X^{h\Gamma}) = \{(x,\phi)\in \Ob(X) \times \Mor(X) \mid \phi\in \Aut(x), \phi^2 = \id\}. \]
Morphisms $(x,\phi)\to (x_1,\phi_1)$ are given by $\alpha\in\Hom(x,x_1)$ such that $\phi_1\alpha=\alpha\phi$.
\end{example}

\begin{example}Consider the $\Gamma$-action on $X\times X$ given by $\overline{(x_0,x_1)} = (x_1,x_0)$. Then $(X\times X)^{h\Gamma}$ is canonically isomorphic to the groupoid with objects:
    \[ \{(x, y,\phi)\in \Ob(X)\times\Ob(X)\times\Mor(X) \mid \phi\in\Hom(x,y) \}, \]
    and arrows $(x,y,\phi)\to (x_1,y_1,\phi_1)$ given by pairs of morphisms $\alpha\colon x\to x_1$ and $\beta\colon y\to y_1$, such that $\phi_1\alpha = \beta\phi$. The map $X\to (X\times X)^{h\Gamma}$, $x\mapsto (x,x,\id)$, is a full and faithful functor. Further, any $(x,y,\phi)\in (X\times X)^{h\Gamma}$ is isomorphic to $(x,x,\id)$ via the pair of morphisms $\id\colon x\to x$ and $\phi\colon x\to y$. Hence, we have a canonical weak equivalence $X \simeq (X\times X)^{h\Gamma}$.
\end{example}

Define $\i\colon X^{h\Gamma} \to X$ by $\i((x,\phi)) = x$. Map an arrow $(x,\phi)\to (x_1,\phi_1)$ in $X^{h\Gamma}$ to the morphism $\alpha\colon x\to x_1$ it uniquely corresponds to.
\begin{prop}\label{prop:grpdincfib}The map $\i\colon X^{h\Gamma}\to X$ is a fibration. 
\end{prop}

\begin{proof}
    Let $(x,\phi)\in X^{h\Gamma}$. Let $\alpha\colon x\to x_1$ be a morphism in $X$. Then
    \[
        \overline{\bar\alpha\phi\alpha^{-1}} = \alpha\bar\phi\bar\alpha^{-1} 
                                             = \alpha \phi^{-1}\bar\alpha^{-1} 
                                             =(\bar\alpha\phi\alpha^{-1})^{-1}.
                                         \]
    Hence, $(x_1, \bar\alpha\phi\alpha^{-1})\in X^{h\Gamma}$. Further, $\alpha$ gives a morphism $(x,\phi)\to (x_1, \bar\alpha\phi\alpha^{-1})$ which is mapped to $\alpha$ under $\i$. Thus, $\i$ is a fibration. 
\end{proof}

Let $Y$ be another groupoid equipped with a $\Gamma$-action. 
A $\Gamma$-equivariant map $f\colon X\to Y$ induces a map $f^{h\Gamma}\colon X^{h\theta}\to Y^{h\Gamma}$ via the prescription:
\[
    f^{h\Gamma}((x,\phi)) = (f(x), f(\phi)) \qquad \text{and}\qquad
f^{h\Gamma}(\alpha) = f(\alpha).
\]
Here, the label $\alpha$ has been used for both the arrow $(x,\phi)\to (x_1,\phi_1)$, as well as the corresponding morphism $x\to x_1$ in $X$ that uniquely determines it.

\begin{prop}\label{prop:grpdfib}If $f$ is a fibration, then so is $f^{h\Gamma}$.
\end{prop}

\begin{proof}Let $(x,\phi)\in X^{h\Gamma}$. Suppose $(f(x),f(\phi))\to (y,\psi)$ is an arrow in $Y^{h\Gamma}$.
    Let $\alpha\colon f(x)\to y$ denote the corresponding morphism in $Y$. If $f$ is a fibration, then there exists $\beta\colon x\to x_1$ in $X$, such that $f(\beta) = \alpha$. Further,
    \[ \overline{(\bar\beta\phi\beta^{-1})} = \beta\bar\phi\bar\beta^{-1} = \beta\phi^{-1}\bar\beta^{-1} = (\bar\beta\phi\beta^{-1})^{-1}.\]
    So $(x_1,\bar\beta\phi\beta^{-1}) \in X^{h\Gamma}$. Now $\beta\colon x\to x_1$ determines a morphism $(x,\phi)\to (x_1,\bar\beta\phi\beta^{-1})$ that is mapped by $f^{h\Gamma}$ to the arrow $(f(x),f(\phi))\to (x_1,\psi)$ that we started with.
\end{proof}

\begin{prop}\label{prop:grpdweq}If $f$ is a weak equivalence, then so is $f^{h\Gamma}$.
\end{prop}
\begin{proof}We argue that the functor $f^{h\Gamma}$ is full, faithful and essentially surjective. Start with faith. It suffices to show that $f^{h\Gamma}$ is injective on automorphism groups. So suppose given an automorphism of $(x,\phi)$ in $X^{h\Gamma}$. Write $\alpha\colon x\to x$ for the corresponding morphism in $X$. If $f^{h\Gamma}$ maps the former to the identity, then by definition $f(\alpha) = \id$. As $f$ is faithful, this yields $\alpha=\id$. Hence, $f^{h\Gamma}$ is faithful.

    Now we show $f^{h\Gamma}$ is full. Let
    $(f(x), f(\phi)) \to (f(x_1), f(\phi_1))$
    be an arrow in $Y^{h\Gamma}$, with $(x, \phi)$ and $(x_1,\phi_1)$ in $X^{h\Gamma}$. Let $\beta\colon f(x) \to f(x_1)$ be the corresponding morphism in $Y$. As $f$ is full, there exists $\alpha\colon x\to x_1$ such that $f(\alpha)=\beta$. Further, 
    \[
    f(\phi_1\alpha) = f(\phi_1)\beta 
                    = \bar\beta f(\phi) 
                = f(\bar\alpha \phi). \]
    The first equality is by definition of $\alpha$, the second because $\beta$ corresponds to a morphism in $Y^{h\Gamma}$, and the third by equivariance of $f$. As $f$ is faithful, this implies
    $\phi_1\alpha = \bar\alpha \phi$. So $\alpha$ yields an arrow $(x,\phi)\to (x_1,\phi_1)$ in $X^{h\Gamma}$ which, by construction, is mapped by $f^{h\Gamma}$ to the arrow $(f(x),f(\phi)) \to (f(x_1), f(\phi_1))$ that we started with.
    Thus, $f^{h\Gamma}$ is full.
    
    To show $f^{h\Gamma}$ is essentially surjective, let $(y,\psi)\in Y^{h\Gamma}$. As $f$ is essentially surjective, there exists an isomorphism $\alpha\colon f(x) \to y$, for some $x\in X$.
    Now
    \[
    \overline{\bar\alpha^{-1}\psi\alpha} = \alpha^{-1}\bar\psi\bar\alpha 
                                         = \alpha^{-1}\psi^{-1}\bar\alpha 
                                         = (\bar\alpha^{-1}\psi\alpha)^{-1}.
                                     \]
    Hence, 
    $(f(x), \bar\alpha^{-1}\psi\alpha)\in Y^{h\Gamma}$.
    As $\bar\alpha^{-1}\psi\alpha$ is a morphism $f(x)\to f(\bar x)$ and $f$ is full, there exists $\phi\colon x\to \bar x$ such that 
    $f(\phi)=\bar\alpha^{-1}\psi\alpha$.
    As $f$ is faithful, $\bar\phi = \phi^{-1}$. In other words, $(x,\phi)\in X^{h\Gamma}$. Now
    $f^{h\Gamma}((x,\phi)) = (f(x), \bar\alpha^{-1}\psi\alpha)$.
    Moreover, $\alpha\colon f(x)\to y$ gives a morphism 
    $(f(x),\bar\alpha^{-1}\psi\alpha) \to (y, \psi)$
    in $Y^{h\Gamma}$.
    Thus, $f$ is essentially surjective.
\end{proof}

\begin{example}\label{eg:BGfix} Let $G$ be a group and suppose $\Gamma=\ZZ/2\ZZ$ acts on $G$. Then $\Gamma$ acts on $\B G$ in the evident way. By definition,
\[ \Ob((\B G)^{h\Gamma}) = \{ \sigma\in G \mid \sigma\bar\sigma =1\},\]
with morphisms in $(\B G)^{h\Gamma}$ given by an arrow $\sigma\to \bar g\sigma g^{-1}$, for each $g\in G$.

Let $H^1(\Gamma; G)$ be the set
\[ Z^1(\Gamma; G) = \{\sigma\in G \mid \sigma\bar\sigma =1\},\]
modulo the relation $\sigma \sim \bar g\sigma g^{-1}$, for all $g\in G$. Then $H^1(\Gamma; G)$ parametrizes isomorphism classes in $(\B G)^{h\Gamma}$. The automorphism group of $\sigma\in (\B G)^{h\Gamma}$ is
\[ K_{\sigma} = \{g\in G \mid \bar g \sigma  = \sigma g\}.\]
Consequently, we have a weak equivalence:
\[ (\B G)^{h\Gamma} \simeq \bigsqcup_{[\sigma]\in H^1(\Gamma; G)} \B K_{\sigma}.\]

Any groupoid is weakly equivalent to one of the form $\bigsqcup_i \B G_i$.\footnote{
There is no preferred choice - it depends on picking an object in each isomorphism class of the groupoid.}
In principle, the formula above gives an explicit presentation (up to weak equivalence) of the homotopy fixed points for any groupoid equipped with a $\Gamma$-action.
\end{example}

\begin{example}\label{eg:parspace}
    Let $G$ be a group, $\theta\colon G\to G$ an involution, and $B\subset G$ a $\theta$-stable subgroup (i.e., $\theta(B) = B$). Let $B\times B$ act on $G$ via $(b_1,b_2)\cdot g = b_1gb_2^{-1}$. Define an action of $\Gamma=\ZZ/2\ZZ$ on $\E_{B\times B} G$ via the following prescription:
    \[ \bar g = \theta(g^{-1}), \quad g\in G; \qquad \overline{(b_1,b_2)} = (\theta(b_2), \theta(b_1)), \quad (b_1,b_2)\in B\times B.\]
    Set
    \[ Z = \{g\in G \mid g\theta(g) = 1\}, \]
    and let $B$ act on $Z$ via twisted conjugation:
    \[ b\cdot g = \theta(b)gb^{-1}. \]
    By definition, $(\E_{B\times B}G)^{h\Gamma}$ is the groupoid $\E_{B\times B}X$, where
    \[ X = \{(b_1,b_2, g) \in B \times B \times G \mid b_1gb_2^{-1} = \theta(g^{-1}) \text{ and } b_1\theta(b_2) = 1\}. \]
    and $B\times B$ acts on $X$ via
    \[ (\beta_1,\beta_2)\cdot (b_1,b_2,g) = (\theta(\beta_2)b_1\beta_1^{-1}, \theta(\beta_1)b_2\beta_2^{-1}, \beta_1 g \beta_2^{-1}). \]
    Let
    \[ Y = \{(b,g)\in B\times G \mid bg\theta(bg)= 1\}, \]
    with $B\times B$-action on $Y$ given by
    \[ (\beta_1, \beta_2)\cdot (b,g) = (\theta(\beta_2)b\beta_1^{-1}, \beta_1g\beta_2^{-1}). \]
    Then the map $X\to Y$ given by
    $(b_1,b_2,g) \mapsto (b_1,g)$
    is an isomorphism of $B\times B$-sets (the inverse is $(b,g)\mapsto (b,\theta(b^{-1}), g)$). This induces a canonical isomorphism of groupoids $\E_{B\times B}X \cong \E_{B\times B}Y$.
Further, the map $Y\to Z$, $(b,g) \mapsto bg$, induces an acyclic fibration $\E_{B\times B}Y \to \E_B Z$, since the action of $1\times B$ on $Y$ is free. Putting everything together, we obtain a canonical acyclic fibration:
\[ (\E_{B\times B}G)^{h\Gamma} \mapright{\sim} \E_B Z. \]
\end{example}
\subsection{Filtered colimits}
All (small) limits and colimits exist in the category of groupoids. Limits are given by limits of sets on objects and morphisms. Colimits can be nebulous to describe explicitly. However, \emph{filtered} colimits present no difficulties. Filtered colimits of groupoids are given by filtered colimits of sets on objects and morphisms.

Let $X$ be a functor from a category $I$ to groupoids with $\Gamma$-action. That is, $X$ is a functor from $I$ to groupoids, such that:
(a) for each $i\in I$, the groupoid $X(i)$ is equipped with a $\Gamma$-action;
(b) for each morphism $i\to j$ in $I$, the morphism $X(i)\to X(j)$ is $\Gamma$-equivariant.
Any object $x_i$ (resp. morphism $\phi_i$) in $X(i)$ defines an object $[x_i]$ (resp. morphism $[\phi_i]$) in $\varinjlim_I X$. If $(x_i,\phi_i) \in X(i)^{h\Gamma}$, then $([x_i],[\phi_i])\in (\varinjlim_I X)^{h\Gamma}$. A similar statement applies to arrows in $X(i)$ that correspond to morphisms in $X(i)^{h\Gamma}$. Consequently, we have a canonical map $\varinjlim_I X^{h\Gamma} \to (\varinjlim_I X)^{h\Gamma}$.
\begin{prop}\label{prop:colimitcommute}If $I$ is filtered, then the canonical map
    \[ \varinjlim_I X^{h\Gamma} \to (\varinjlim_I X)^{h\Gamma} \]
    is an isomorphism.
\end{prop}

\begin{proof}
    The inverse is given as follows. Let $(x,\phi)\in (\varinjlim_I X)^{h\Gamma}$. Then there exist $x_i\in \Ob(X(i))$ and $\phi_j\in \Mor(X(j))$, for some $i,j\in I$, such that $x$ is the image of $x_i$ and $\phi$ is the image of $\phi_j$.
    As $I$ is filtered we may assume $i=j$, and that $\phi_i$ is a morphism $x_i\to \bar x_i$ such that $\bar\phi_i = \phi_i^{-1}$. 
    In other words, $(x_i,\phi_i)$ defines an object of $X(i)^{h\Gamma}$, and hence an object of $\varinjlim_I X^{h\Gamma}$. This latter object does not depend on any of the choices made, again because $I$ is filtered.
    Similarly, any morphism in $(\varinjlim_I X)^{h\Gamma}$ comes from an arrow in some $X(i)$. As $I$ is filtered, we may assume that this arrow corresponds to a morphism in $X(i)^{h\Gamma}$. Using that $I$ is filtered once again, we see that this defines a morphism in $\varinjlim_I X^{h\Gamma}$ that is independent of the choices made.
\end{proof}

\begin{example}
    Our construction/model of homotopy fixed points of a groupoid is a special case of a more general version for simplicial sets. One goes back and forth between simplicial sets and groupoids using the nerve and fundamental groupoid constructions. However, Proposition \ref{prop:stalkscommute} does \emph{not} hold for simplicial sets in general. Here is a standard counterexample.
    
    It will be convenient to conflate simplicial sets with their topological realization. For topological spaces $X,Y$, write $\Map(X,Y)$ for the set of continuous maps $X\to Y$. Topologize $\Map(X,Y)$ using the compact open topology. Then the set of path components $\pi_0\Map(X,Y)$ consists of homotopy classes of maps. 
    
    Let $\RR\PP^n$ denote real projective space of dimension $n$. If $\Gamma$ acts trivially on $Y$, then $Y^{h\Gamma} = \Map(\RR\PP^{\infty}, Y)$ (this can be taken as a definition).
    Let $\Gamma$ act on each $\RR\PP^n$, $n\in\ZZ_{\geq 0}$, trivially. Then
    \[ \varinjlim_n (\RR\PP^n)^{h\Gamma} = \varinjlim_n \Map(\RR\PP^{\infty}, \RR\PP^n) \qquad
    \text{and}\qquad (\varinjlim_n \RR\PP^n)^{h\Gamma} = \Map(\RR\PP^{\infty}, \RR\PP^{\infty}).\]
    However,
    \[ \pi_0\varinjlim_n \Map(\RR\PP^{\infty}, \RR\PP^n) \neq \pi_0\Map(\RR\PP^{\infty}, \RR\PP^{\infty}). \]
    Indeed, as each $\RR\PP^n$ is Hausdorff, so is $\Map(\RR\PP^{\infty}, \RR\PP^n)$. Consequently,
    \[ \pi_0\varinjlim_n \Map(\RR\PP^{\infty}, \RR\PP^n) = \varinjlim_n\pi_0\Map(\RR\PP^{\infty}, \RR\PP^n).\]
    Now the identity $\RR\PP^{\infty} \to\RR\PP^{\infty}$ is not homotopic to any map that factors through $\RR\PP^n$, $n<\infty$, since $\RR\PP^{\infty}$ has non-vanishing cohomology in arbitrarily high degree. 
    Thus, $\varinjlim_n (\RR\PP^n)^{h\Gamma}$ is not even weakly equivalent to $(\varinjlim_n \RR\PP^n)^{h\Gamma}$.
\end{example}

\section{(Pre)sheaves of groupoids}
Let $\cC$ be a Grothendieck site. Throughout, we assume $\cC$ has enough points (see \cite[Tag 00YJ]{Stacks}). By definition of a point, the stalk of a presheaf of sets at a point is a left exact functor to the category of sets. On the other hand, stalks are defined as a certain colimit of sets (see \cite[Tag 00Y3]{Stacks}). Now a (small) category is filtered if and only if colimits into the category of sets, indexed by said category, commute with finite limits \cite[Theorem 3.1.6]{KS}. It follows that stalks are given by filtered colimits.
Set\footnote{`presheaf' = `strict contravariant functor' (as opposed to a lax functor).}
\[ \PreGrpd(\cC) = \text{category of presheaves of groupoids on $\cC$}.\] 
For a point $t$ of $\cC$, and $X\in \PreGrpd(\cC)$, write $X_t$ for the stalk at $t$ (defined using the same formula as for presheaves of sets). Each stalk $X_t$ is a groupoid.

    \begin{example}
        Let $S$ be a Hausdorff (more generally, a sober) topological space. The site $\Op|_S$ is the poset of open subsets $U\subset S$. A covering family for an open subset $U$ is an open cover $U_i \subset U$. Each point $t\in S$, in the usual sense, gives a point of the site $\Op|_S$. Every point of $\Op|_S$ is obtained this way, and $\Op|_S$ has enough points. For a presheaf $X$ on $\Op|_S$, the stalk at $t$ is defined in the usual way by
$F_t = \varinjlim_{t\in U} X(U)$,
This colimit is clearly filtered.
\end{example}

\begin{example}Let $S$ be a scheme. The \'etale site $\et|_S$ has as objects all \'etale maps $\phi\colon U\to S$. Morphisms are given by the evident commutative triangles.
    The coverings are surjective families of \'etale morphisms over $S$. If $t\colon \Spec(k) \to S$ is a morphism from an algebraically closed field $k$, then $t$ defines a point of $\et|_S$. Every point of $\et|_S$ is obtained this way and $\et|_S$ has enough points. For a presheaf $X$ on $\et|_S$, the stalk at $t$ is $X_t = \varinjlim_{(U,t)} X(U)$,
                            where the colimit is indexed by diagrams
                            $\Spec(k) \to U \to S$,
                            with $U\to S$ \'etale, and such that the composite map $\Spec(k) \to S$ is $t$.
                            These diagrams are the \'etale neighborhoods of $t$. As $\et|_S$ has all finite products (fibre products over $S$), the category of \'etale neighborhoods of $t$ is cofiltered.
                        \end{example}

                        \begin{example}Let $S$ be a scheme. The big \'etale site on $S$ has as underlying category the category of all schemes over $S$. The coverings are surjective families of \'etale morphisms over $S$. This site has enough points. Points are given by points $t$ of $\et|_T$ for every object $T$ of the big site. The stalk at $t$ of a presheaf $X$ is defined by restricting to $\et|_T$. I.e., $X_t = (X|_{\et|_T})_t$.
                            See \cite[Tag 03PN]{Stacks} for more details.
                        \end{example}

A morphism $X\to Y$ in $\Pre\Grpd(\cC)$ is called a \emph{sectionwise weak equivalence} (resp. \emph{sectionwise fibration}) if it induces a weak equivalence (resp. fibration) of groupoids $X(U) \to Y(U)$, for all $U\in \cC$. It is a \emph{local weak equivalence} (resp. \emph{local fibration}) if it is a weak equivalence (resp. fibration) of groupoids on all stalks.
Fibrations that are also weak equivalences are called \emph{acyclic}.
Sectionwise weak equivalences (resp. sectionwise fibrations) are local weak equivalences (resp. local fibrations).

\begin{example}A morphism between sheaves of sets (viewed as groupoids) is a local weak equivalence if and only if it is an isomorphism.
\end{example}

    Extending the usual notion of sheafification to presheaves of groupoids poses no real difficulties. The canonical map from a presheaf of groupoids to its sheafification is a local weak equivalence.
A stack is a sheaf of groupoids for which descent data relative to any covering is effective.
 A local weak equivalence between stacks is always a sectionwise weak equivalence (effective descent). In other words, a local weak equivalence between stacks coincides with the traditional notion of `an equivalence of stacks'.
    Similar to sheafification, the forgetful functor from stacks to presheaves of groupoids admits a left adjoint called stack completion (or stackification). For $X\in\PreGrpd(\cC)$, set
    \[ [X]= \text{stack completion of $X$}. \]
    An explicit model for $[X]$ may be constructed using a process similar to sheafification. The unit of adjunction yields a canonical map $X\to [X]$ which is always a local weak equivalence. A convenient reference is \cite[Ch. 9]{J} (or see \cite[Tag 02ZO]{Stacks}).

    \begin{example}Let $G$ be a presheaf of groups on $\cC$. Let $X$ be a presheaf of sets on $\cC$ equipped with a $G$-action. Define $E_GX\in \PreGrpd(\cC)$ by:
    \[ (E_GX)(U) = \E_{G(U)}X(U).\]
    Set
    \[ EG = E_G G \qquad \mbox{and}\qquad BG = E_G(\ast).\]
    Define $[X/G]$ by:
    \begin{align*}
        [X/G](U) = \{(P, \varphi)\mid &\, P\text{ is a $G$-torsor over $U$}, \varphi\colon P\to X|_U \\
                                      &\, \text{is a $G|_U$-equivariant morphism}\}
        \end{align*}
        Finally, set
        \[ [BG] = [\ast/G]. \]
        Then the canonical map $EG \to \ast$ is a sectionwise weak equivalence. The evident map $BG\to [BG]$ (map the single object in $BG(U)$ to the trivial torsor on $U$) is a local weak equivalence. It is generally not a sectionwise weak equivalence. The presheaf of groupoids $E_GX$ is usually not a stack. In particular, $BG$ is generally not a stack. However, $[X/G]$ and in particular $[BG]$ is a stack. In fact, $[X/G]$ is the stack completion of $E_GX$.
    \end{example}

    \subsection{Homotopy fixed points}
    Suppose $\Gamma=\ZZ/2\ZZ$ acts on $X\in\PreGrpd(\cC)$. That is, $\Gamma$ acts on sections $X(U)$, $U\in \cC$, and these actions are compatible with restriction maps. Define $X^{h\Gamma}$ by:
\[ X^{h\Gamma}(U) = X(U)^{h\Gamma}.\]
The map of groupoids on sections:
\[ X^{h\Gamma}(U) \to X(U), \qquad (x,\phi)\mapsto x, \]
$U\in \cC$, yields a morphism $\iota\colon X^{h\Gamma} \to X$.
\begin{prop}The morphism $\iota\colon X^{h\Gamma} \to X$ is a sectionwise fibration.
\end{prop}

\begin{proof}Apply Proposition \ref{prop:grpdincfib}.
\end{proof}

\begin{prop}\label{prop:stalkscommute}Let $t$ be a point of $\cC$.
    Then there is a canonical isomorphism of groupoids $(X^{h\Gamma})_t \cong (X_t)^{h\Gamma}$.
\end{prop}

\begin{proof}Stalks are given by filtered colimits. Apply Proposition \ref{prop:colimitcommute}.
%
\end{proof}

Let $Y$ be another presheaf of groupoids equipped with a $\Gamma$-action. A $\Gamma$-equivariant morphism $f\colon X\to Y$ is a morphism in $\PreGrpd(\cC)$ that commutes with the $\Gamma$-actions on sections. Such an $f$ induces a morphism $f^{h\Gamma}\colon X^{h\Gamma} \to Y^{h\Gamma}$.

\begin{cor}
    If $f$ is a local fibration, then so is $f^{h\Gamma}$.
\end{cor}
\begin{proof}
    Apply Proposition \ref{prop:grpdfib}.
\end{proof}

    \begin{cor}
    If $f$ is a local weak equivalence, then so is $f^{h\Gamma}$.
\end{cor}
\begin{proof}Apply Proposition \ref{prop:grpdweq}.\end{proof}

A $\Gamma$-action on $X\in\PreGrpd(\cC)$ induces a $\Gamma$-action on the stack completion $[X]$ such that the canonical morphism $X\to [X]$ is equivariant.  
\begin{cor}\label{cor:stackequiv}There is a canonical local weak equivalence $X^{h\Gamma}\mapright{\sim} [X]^{h\Gamma}$.
\end{cor}
\begin{proof}The canonical map $X \to [X]$ is a local weak equivalence.\end{proof}


\section{Langlands parameter spaces}
Let $G$ be a presheaf of groups on $\cC$, equipped with an involution $\theta\colon G\to G$. Let $B\subset G$ be a subgroup presheaf that is stable under $\theta$. Consider the $B\times B$-action on $G$ given by:
\[ (b_1,b_2)\cdot g = b_1gb_2^{-1}.\]
Then $\Gamma=\ZZ/2\ZZ$ acts on the presheaf of groupoids $E_{B\times B}G$ via 
\[\bar g = \theta(g^{-1}),\]
for $g\in \Ob(E_{B\times B}G(U))$, and
\[\overline{(b_1,b_2)} = (\theta(b_2), \theta(b_1)),\]
for $(b_1,b_2)\in \Mor(E_{B\times B}G(U))$, $U\in \cC$. Consider the presheaf of sets $Z^1(\theta; G)$ given by
\[ Z^1(\theta; G)(U) = \{ g\in G(U) \mid g\theta(g) = 1\}.\]
Then $B$ acts on $Z^1(\theta; G)$ via twisted conjugation:
\[ b\cdot g = \theta(b)gb^{-1}.\]
\begin{prop}
    There is a canonical acyclic sectionwise fibration:
    \[ (E_{B\times B}G)^{h\Gamma} \mapright{\sim} E_B Z^1(\theta; G).\]
\end{prop}

\begin{proof}Immediate from Example \ref{eg:parspace}.
\end{proof}

As a special case, let $\cC$ be the big \'etale site of complex varieties\footnote{`complex variety' = `separated scheme of finite type over $\Spec(\CC)$'}. Let $G$ be a complex reductive group, $\theta$ an involution of $G$ that preserves some Borel subgroup $B\subset G$. Write $B\backslash G$ for the corresponding flag variety.
\begin{cor}
    There is a canonical equivalence of stacks:
    \[ [B\backslash G/B]^{h\Gamma} \mapright{\sim} [Z^1(\theta; G)/B]. \]
    \end{cor}
    \begin{proof}
        Apply Corollary \ref{cor:stackequiv}.
    \end{proof}
    This is the promised identification of Langlands parameter spaces (trivial infinitesimal character). See \cite{S} for a streamlined exposition of the significance of this space to representation theory.

\end{document}